\documentclass[runningheads,envcountsame,a4paper]{llncs} 

\usepackage{amssymb,amsmath}
\usepackage{wasysym}
\usepackage{graphicx}
\usepackage{epsfig}
\usepackage{latexsym}
\usepackage[english]{babel}
\usepackage[utf8]{inputenc}
\usepackage{csquotes}
\usepackage{algorithm} 
\usepackage{algpseudocode}
\usepackage{tikz}
\usetikzlibrary{arrows,shapes.geometric,positioning,calc}
\usepackage[colorinlistoftodos]{todonotes}
\usepackage{changepage}
\usepackage{subfig}
\usepackage{soul}
\usepackage{hyperref}
\usepackage{todonotes}
\usepackage{enumerate}

\newtheorem{thm}{Theorem}
\newtheorem{prop}[thm]{Proposition}
\newtheorem{cor}[thm]{Corollary}
\newtheorem{lem}[thm]{Lemma}

\newtheorem{exa}[thm]{Example}

\newsavebox{\qedB}

\sbox{\qedB}{\setlength{\unitlength}{1mm}
 \begin{picture}(4,4)(0,0)
  \thinlines
  {\put(0,0){\framebox(2.83,2.83){}}}%
  {\put(1.17,1.17){\framebox(2.83,2.83){}}}%
  {\put(0,0){\framebox(4,4){}}}%
  {\put(1.17,1.17){{\rule{1ex}{1ex} }}}%
 \end{picture}}
 
\newcommand{\QEDB}{\ifmmode\def\next{\tag"\usebox{\qedB}"}%
 \else\let\next=\relax
 {\unskip\nobreak\hfil\penalty50
 \hskip2em\hbox{}\nobreak\hfil\usebox{\qedB}
 \parfillskip=0pt \finalhyphendemerits=0\penalty-100\bigskip}\fi\next}

\newcommand{\Alphabet}{\hbox{\rm Alph}}

\newcommand{\Pref}{\hbox{\rm Pref}}
\newcommand{\Suff}{\hbox{\rm Suff}}

\newcommand{\fac}{\hbox{\rm Fac}}

\newcommand{\N}{\mathbb N}

\newcommand{\prim}{\hbox{\rm Prim}}
\newcommand{\bprop}{\begin{prop}}
\newcommand{\eprop}{\end{prop}}
\newcommand{\bcor}{\begin{cor}}
\newcommand{\ecor}{\end{cor}}
\newcommand{\blem}{\begin{lem}}
\newcommand{\elem}{\end{lem}}

\title{On the number of $k$-powers in a finite word}
\toctitle{Square complexity}
\author{Shuo LI\inst{1}}

\institute{Laboratoire de Combinatoire et d'Informatique Mathématique,\\
Université du  Québec \`a Montréal,\\
CP 8888 Succ. Centre-ville, Montréal (QC) Canada H3C 3P8\\
\email{shuo.li.ismin@gmail.com} 
}

\begin{document}

\maketitle


\begin{abstract}

This note is an attempt to attack a conjecture of Fraenkel and Simpson stated in 1998 concerning the number of distinct squares in a finite word. By counting the number of (right-)special factors, we give an upper bound of the number of {\em $k$-powers} in a finite word for any integer $k\geq 3$. By {\em $k$-power}, we mean a word of the form $\underbrace{uu...u}_{k \; \text{times}}$. 

\end{abstract}

\section{Introduction and notation}
Given a finite word, the problem of counting the number of distinct squares was introduced by Fraenkel and Simpson. In~\cite{FraenkelS98} they conjectured that the number of distinct squares in
a finite word $w$ is bounded by its length $|w|$ and they proved that this number is bounded by $2|w|$. After that Ilie~\cite{Ilie} strengthened this bound to $2|w|-\Theta(n)$; Lam~\cite{lam}
improved this result to $\frac{95}{48}|w|$; Deza, Franek and Thierry~\cite{DezaFT15}
achieved a bound of $\frac{11}{6}|w|$; Thierry~\cite{thie} refined this bound to $\frac{3}{2}n$. A basic fact about the square-counting problem is that no more than two squares can have their last occurrence starting at the same position, this fact is proved in~\cite{FraenkelS98} using the three squares lemma of Crochemore and Rytter~\cite{crry}. After that, the ideas of improving the bound of distinct squares in a finite word are about counting the number of positions at which there exist two different squares having their last occurrence starting. In this article, instead of studying the same motif, we propose to consider the number of special factors. The idea is from the fact that (almost) every occurrent factor can be associated with a (right-)special factor. Even though we can not achieve an injection from the set of squares to the set of special factors in the finite word. This correspondence leads an upper bound of the number of $k$-powers in the given word. The main result of this note is announced as follows:

\begin{thm}
\label{th:main}
Let $k$ be an integer larger than 2. For any finite word $w$, let $N_k(w)$ denote the number of its distinct non-empty factors of the form $\underbrace{uu...u}_{k\; \text{times}}$, let $|w|$ denote the length of $w$ and let $|\Alphabet(w)|$ denote the number of distinct letters in $w$. We then have $$N_k(w) \leq \frac{|w|-|\Alphabet(w)|}{k-2}.$$
\end{thm}

\section{Preliminaries}

Let us recall the basic terminology about words. By {\em word} we mean a finite concatenation of symbols $w = w_1 w_2 \cdots w_{n}$, with $n \in \N$. The {\em length} of $w$, denoted $|w|$, is $n$ and we say that the symbol $w_i$ is at {\em position} $i$. The set $\Alphabet(w)=\left\{w_i| 1\leq i \leq n\right\}$ is called the {\em alphabet} of $w$ and its elements are called {\em letters}. Let $|\Alphabet(w)|$ denote the cardinality of $\Alphabet(w)$. A word of length $0$ is called the {\em empty word} and it is denoted by $\varepsilon$. For any word $u$, we have $u=\varepsilon u=u\varepsilon$.	Let $u,v$ be two different words, we say $a$ is {\em shorter} (resp. {\em longer}) than $b$ if $|a| < |b|$ (resp. $|a| > |b|$).

A word $u$ is called a {\em factor} of $w$ if $w = pus$ for some words $p,s$.
When $p = \varepsilon$ (resp. $s = \varepsilon$) $u$ is called a {\em prefix} (resp. {\em suffix}) of $w$.
The set of all factors (resp. prefixes, resp. suffixes) of $w$ is denoted by $\fac(w)$ (resp. $\Pref(w)$, resp. $\Suff(w)$). For any integer $i$ satisfying $1 \leq i \leq |w|$, let $C_w(i)$ denote the number of distinct factors of length $i$ in $w$.

A factor $u$ of $w$ is called right-special if there exist two different letters $a, b \in \Alphabet(w)$ such that $ua$ and $ub$ are both factors of $w$. $ua,ub$ are called {\em right-extensions} of $u$.

For any natural number $k$, the {\em $k$-th power} of a finite word $u$ is denoted by $u^k = u u \cdots u$ and consists of the concatenation of $k$ copies of $u$.
A finite word $w$ is said to be {\em primitive} if it is not a power of another word, that is if $w=u^k$ implies $k=1$. Let $\prim(w)$ denote the set of all primitive factors of $w$.
A \emph{$k$-power} is a word $w$ satisfying $w = \underbrace{uu...u}_{k\; \text{tilmes}}$ for a certain $u \in \fac(w)$ and for a $k \geq 2$. Let $N_k(w)$ denote the number of its distinct non-empty $k$-powers. For a given word $w$ and two positive integers $a,b$ satisfying $a \leq b=|w|$, let us define $w^{\frac{a}{b}}$ to be the prefix of $w$ of length $a$. Now we can define the rational power a word: let $w$ be a finite word and let $\frac{a}{b} \in \mathbb{Q}^+$ be a positive rational number, $w^{\frac{a}{b}}$ is well defined only if $b=|w|$, and in this case, there is a couple of non-negative integers $(c,d)$ satisfying $a=c|w|+d$, we define $w^{\frac{a}{b}}$ to be $w^cw^{\frac{d}{b}}$. For a given word $w$ and a given integer $k$, we say that $w$ is of {\em period} $k$ if there exists a word $u$ of length $k$ such that $w=u^{\frac{|w|}{k}}$.\\

Here we recall a basic lemma concerning the repetitions:
\begin{lem}[Fine and Wilf~\cite{fiwi}]
\label{period}
Let $w$ be a word having $k$ and $l$ for periods. If $|w| \geq k+l-\gcd (k,l)$ then $\gcd(k,l)$ is also a period of $w$.
\end{lem}

\section{Number of right-special factors}

Let $w$ be a finite word. In this section, we consider the word $w^*$ obtained by concatenating a special letter $*$ at the end of $w$, with the condition that $*\not \in \Alphabet(w)$. 

For any $u \in \fac(w)$, let us define $m_w(u)=\max\left\{i|u^i \in \fac(w), i \in \mathbb{Q}^+\right\}$, and similarly, let us define $m_{w^*}(u)=\max\left\{i|u^i \in \fac(w^*), i \in \mathbb{Q}^+\right\}$. There are the following facts between $m_w(u)$ and $m_{w^*}(u)$:

1) For any factor $u \in \fac(w)$, $m_w(u)=m_{w^*}(u)\geq 1$;

2) If $u \in \fac(w^*)$ but $u \not \in \fac(w)$, $m_{w^*}(u)=1$.

For any factor $u \in \fac(w)$, let us define $m(u)=m_{w}(u)=m_{w^*}(u)$. For any integer $i$ satisfying $1 \leq i \leq m(u)$, let us define $u(i)$ to be the shortest suffix of $u^{m(u)}$ containing $u^i$ as a prefix.

\begin{exa}
Let us consider the following word $$w=abababacababa.$$ For $u=ab$,$m(u)=\frac{7}{2}$, thus $u(1),u(2),u(3)$ are all well-defined, we have $u(1)=aba$, $u(2)=ababa$ and $u(3)=abababa$. 
For $v=abab$, we then have $m(v)=\frac{7}{4}$, thus only $v(1)$ is well-defined, we have $v(1)=ababa$.
\end{exa} 

\begin{lem}
\label{special}
Let $w$ be a finite word and let $u\in \fac(w)$, if $m(u)\geq 2$, then, for any integer $i$ satisfying $1 \leq i \leq m(u)-1$, the factor $u(i)$ is a right-special factor in $w^*$. 
\end{lem}

\begin{proof}
Let $w=w_1w_2...w_{|w|}$ and let $w^*=w^*_1w^*_2...w^*_{|w|+1}$, with $w_i=w^*_i$ for all $i$ satisfying $1 \leq i \leq |w|$ and $w_{|w|+1}=*$. As $u^{m(u)} \in \fac(w)$, there exists an integer $k$ satisfying $1 \leq k \leq |w|$ such that $u^{m(u)}=w_kw_{k+1}...w_{k+m(u)|u|-1}$. For any integer $i$ satisfying $1 \leq i \leq m(u)-1$, $u(i)$ occurs at least twice in $u^{m(u)}$ as respectively a prefix and a suffix of $u^{m(u)}$. If we suppose that $|u(i)|=l$, then $$u(i)=w_kw_{k+1}...w_{k+l-1}=w_{k+m(u)|u|-l-1}w_{k+m(u)|u|-l}...w_{k+m(u)|u|-1}.$$ Now let us prove that $u(i)$ is a right-special factor of $w^*$. As $u^{m(u)} \in \fac(w)$, $w^*_kw^*_{k+1}...w^*_{k+m(u)|u|-1}w^*_{k+m(u)|u|}$ is well-defined and it is a factor of $w^*$. We claim that $w^*_{k+l} \neq w^*_{k+m(u)|u|}$. In fact, if $k+m(u)|u|=|w|+1$, then $w^*_{k+m(u)|u|}=* \neq w^*_{k+l}$; if $k+m(u)|u|<|w|+1$, 
$$w^*_kw^*_{k+1}...w^*_{k+m(u)|u|-1}w^*_{k+m(u)|u|}=w_kw_{k+1}...w_{k+m(u)|u|-1}w_{k+m(u)|u|} \in \fac(w),$$
in this case if $w_{k+l-1} = w_{k+m(u)|u|}$, then $w_kw_{k+1}...w_{k+m(u)|u|-1}w_{k+m(u)|u|}=u^{m(u)+\frac{1}{|u|}}$, this contradicts the maximality of $m(u)$. As a consequence, $w^*_{k+l-1} \neq w^*_{k+m(u)|u|}$ in both cases, thus $$w^*_kw^*_{k+1}...w^*_{k+l}\neq w^*_{k+m(u)|u|-l-1}w^*_{k+m(u)|u|-l}...w^*_{k+m(u)|u|}.$$ Hence, $u(i)$ is a right-special factor in $w^*$.
\end{proof}

\begin{exa}
Let us consider the word given in the previous example: $$w=abababacababa.$$ $w^*$ is to be $$w^*=abababacababa*.$$ For $u=ab$,$m(u)=\frac{7}{2}$, thus $u(1),u(2)$ are right-special factors in $w^*$. In fact, $(ab)^iab$, $(ab)^iac$ and $(ab)^ia*$ are all factors of $w^*$ for $i=1,2$. 
For $v=abab$, $m(v)=\frac{7}{4}$. $v(1)$ is also a right-special factor of $w^*$ because $v(1)=ababa=u(2)$. However, it is not counted in the lemma, because $m(v)-1 <1$.
\end{exa}

\begin{lem}
\label{prim}
For any couple of different primitive factors $(u , v) \in \fac(w)^2$ satisfying $m(u), m(v) \geq 3$, we have $$\left\{u(i)| 2 \leq i \leq m(u)-1\right\} \cap \left\{v(i)| 2 \leq i \leq m(v)-1\right\}=\emptyset $$. 
\end{lem}

\begin{proof}
If there exist two different primitive factors $u,v$ and two integers $i,j$ such that $u(i)=v(j)$ then $u^iu'=v^jv'$ with $u',v'$ respectively a prefix of $u$ and $v$. From the hypothesis that $i\geq 2,j \geq 2$ and Lemma~\ref{period}, there exists a factor $p$ such that $u,v$ are both a power of $p$, this fact contradicts the primitivities of $u,v$.
\end{proof}

\begin{cor}
\label{number}
Let $w$ be a finite word and let $M(w^*)$ denote the number of right-special factors in $w^*$, then we have
$$\sum_{u \in \prim(w)} (m(u)-2) \leq M(w^*).$$ 
\end{cor}

\begin{proof}
Let us consider the set of factors of $w$: $$s=\left\{u(i)|u \in \prim(w), 2 \leq i \leq m(u)-1, u(i) \in \fac(w) \right\}.$$ From Lemma~\ref{special}, the elements in $s$ are all right-special factors in $w^*$ and from Lemma~\ref{prim}, the cardinality of $s$ is exactly  $\sum_{u \in \prim(w)} (m(u)-2)$.
\end{proof}

\section{Proof of the main theorem}

\begin{proof}[ of Theorem~\ref{th:main}]

Let $k$ be an integer larger than $2$, we have
$$N_k(w)=\sum_{u \in \prim(w)}\lfloor \frac{m(u)}{k}\rfloor,$$ where $\lfloor x \rfloor$ represents the largest integer smaller than or equal to $x$. On the other hand, for any primitive $u$ satisfying $m(u) \geq k$, we have $\frac{m(u)}{k}\leq \frac{m(u)-2}{k-2}$. Hence,

\begin{align*}
N_k(w)&=\sum_{u \in \prim(w)}\lfloor \frac{m(u)}{k}\rfloor\\
&=\sum_{\substack{u \in \prim(w)\\m(u) \geq k}}\lfloor \frac{m(u)}{k}\rfloor\\
&\leq \sum_{\substack{u \in \prim(w)\\m(u) \geq k}} \frac{m(u)-2}{k-2}\\
&\leq \sum_{u \in \prim(w) }\frac{m(u)-2}{k-2}\\
&\leq \frac{M(w^*)}{k-2},
\end{align*}
where $M(w^*)$ denote the number of right-special factors in $w^*$.

Now let us give an upper bound of $M(w^*)$. First, we claim that if $u \in \fac(w^*)$ is a right-special factor, then $u \in \fac(w)$. Otherwise, $u$ is a suffix of $w^*$ and it does not have any right-extensions. Now, for any right-special factor of $w^*$ of length $i$, it has at least two right-extensions, and the suffix of $w^*$ of length $i$ is not right-special. Thus, if we let $s(w^*)(i)$ denote the number of right-special factors of $w^*$ of length $i$, we then have $s(w^*)(i) \leq C_{w^*}(i+1)-C_{w^*}(i)+1$, where $C_{w^*}(i)$ is the number of distinct factors of length $i$ in $w^*$ defined in the section Preliminaries. Hence,
 
\begin{align*}
M(w^*)&=\sum_{i=1}^{|w^*|}s(w^*)(i)\\
&\leq \sum_{i=1}^{|w^*|}C_{w^*}(i+1)-C_{w^*}(i)+1\\
&\leq |w|+1+C_{w^*}(|w^*|+1)-C_{w^*}(1)\\
&\leq |w|-|\Alphabet(w)|.
\end{align*}

Thus, $$N_k(w) \leq \frac{|w|-|\Alphabet(w)|}{k-2}.$$
\end{proof}

\section{Concluding Remarks}
\label{sec:concl}

The result obtained in the main theorem is not sharp. Let us consider the word $w=aaa$, it is easy to check that $N_3(w)=1$ while the bound given in Theorem~\ref{th:main} is $2$. The author believes that the problem is from the way we count the number of right-special factors. From Lemme~\ref{special} we prove that for a given primitive $u$, all $u(i)$ are right-special if $1 \leq i \leq m(u)-1$. However, in Lemma~\ref{prim} we count just the $u(i)$ for $2 \leq i \leq m(u)-1$. In fact, we can only prove that the words of the form $u(i)$ are pairwisely different when $i \geq 2$. Meanwhile, we can have two different primitives $u,v$ such that $u(1)=v(i)$ for some positive integer $i$. Thus, further work to to be done to investigate that $$\sum_{u \in \prim(w)} (m(u)-2) \leq M(w^*).$$
If the previous  inequality holds, we may expect to prove 
$$N_k(w) \leq \frac{|w|-|\Alphabet(w)|}{k-1}.$$


\bibliographystyle{splncs03}
\bibliography{biblio}

\end{document}